\documentclass[10pt,reqno]{amsart} 

\usepackage{amssymb,latexsym}
\usepackage{cite} 

\usepackage[height=190mm,width=130mm]{geometry} 

\theoremstyle{definition}
\theoremstyle{plain}
\usepackage{amsfonts}
\usepackage{amsmath}
\usepackage{amscd}
\usepackage{amsthm}
\usepackage{bbm}
\usepackage{bm}
\usepackage{enumerate}
\usepackage{enumitem}
\usepackage{hyperref}
\usepackage{mathabx}
\usepackage{mathdots}
\usepackage{mathtools}

\usepackage{citehack}
\usepackage{longtable}
\usepackage[matrix,arrow,curve]{xy}

\usepackage{lipsum}

\allowdisplaybreaks

\newcommand{\C}{\mathbb{C}}

\newcommand{\M}{\mathcal{M}}

\newcommand{\wt}{\textup{wt}}

\theoremstyle{plain}
\newtheorem{theorem}{Theorem}
\newtheorem{lemma}{Lemma}

\newtheorem{proposition}{Proposition}

\theoremstyle{definition}

\theoremstyle{remark}

\newcommand{\Z}{\mathbb{Z}}
\newcommand{\V}{\mathcal V}

\numberwithin{equation}{section} 

\newcommand{\E}{\ensuremath{\mathcal{E}}}

\newcommand{\g}{\ensuremath{\Gamma}}

\newcommand{\ps}{{\raise 1pt\hbox{\tiny (}}}

\newcommand{\pss}{{\raise 1pt\hbox{\tiny [}}}
\newcommand{\pdd}{{\raise 1pt\hbox{\tiny ]}}}
\newcommand{\pd}{{\raise 1pt\hbox{\tiny )}}}

\newcommand{\bs}{{\raise 1pt\hbox{\tiny [}}}
\newcommand{\bd}{{\raise 1pt\hbox{\tiny ]}}}

\def\cross{\mathinner{\mathrel{\raise0.8pt\hbox{$\scriptstyle>$}}
                 \joinrel\mathrel\triangleleft}}

\def\W{\mathcal{V}}
\def\W{\mathcal{W}}
\def\K{\mathcal{K}}

\def\Hom{\mathop\text{\rm Hom}\nolimits}

\usepackage{stackrel}

\newcommand{\be}{\begin{equation}}
\newcommand{\ee}{\end{equation}}

\newcommand{\nn}{\nonumber \\}

\newcommand{\nc}{\newcommand}
\nc{\cali}{\mathcal}
\nc{\on}{\operatorname}
\nc{\Wick}{{\mb :}}

\nc{\ddz}{\frac{\partial}{\partial z}}
\nc{\ch}{\mbox{ch}}
\nc{\Oo}{{\cali O}}
\nc{\cond}{|\,}
\nc{\bib}{\bibitem}
\nc{\pone}{\Pro^1}
\nc{\pa}{\partial}
\nc{\arr}{\rightarrow}
\nc{\larr}{\longrightarrow}
\nc{\ket}{)}
\nc{\bra}{(}
\nc{\gam}{\bar{\gamma}}
\nc{\ep}{\epsilon}
\nc{\su}{\widehat{{\mf s}{\mf l}}_2}
\nc{\sw}{{\mf s}{\mf l}}
\nc{\h}{{\mf h}}
\nc{\n}{{\mf n}}
\nc{\ab}{\mf{a}}
\nc{\is}{{\mb i}}
\nc{\js}{{\mb j}}
\nc{\He}{{\cali H}}
\nc{\inv}{^{-1}}
\nc{\ol}{\overline}
\nc{\wh}{\widehat}
\nc{\dst}{\displaystyle}

\nc{\delt}{\partial_t}
\nc{\ddt}{\frac{\partial}{\partial t}}
\nc{\delx}{\partial_x}
\nc{\mb}{\mathbf}
\nc{\mf}{\mathfrak}

\nc{\mbb}{\mathbb}
\nc{\Ctt}{\C((t))}
\nc{\Ct}{\C[t,t\inv]}

\nc{\ghat}{\wh{\g}}

\nc{\un}{\underline}
\nc{\mc}{\mathcal}
\nc{\BB}{{\mc B}}
\nc{\bb}{{\mf b}}
\nc{\kk}{{\mf k}}
\nc{\frob}{\times}
\nc{\sm}{\setminus}
\nc{\Pp}{{\mathbb P}^1}
\nc{\Aa}{{\mc A}}

\nc{\AutO}{\on{Aut}\Oo}
\nc{\AUTO}{\un{\on{Aut}}\Oo}
\nc{\AUTK}{\un{\on{Aut}}\K}
\nc{\Heout}{\He_{\out}}
\nc{\Hetil}{{\widetilde\He}}
\nc{\wb}{\overline}

\nc{\Res}{\on{Res}}
\nc{\pitil}{\Pi}
\nc{\Ctil}{\wt{C}}
\nc{\auto}{\on{Aut} \Oo}
\nc{\phitil}{\wt{\phi}}
\nc{\gz}{\g_{\vec z}}
\nc{\tensorM}{\bigotimes_{i=1}^N{\mathbb M}_i}
\nc{\tensorW}{\bigotimes_{i=1}^N W_{\nu_i,k}}
\nc{\out}{\on{out}}

\nc{\m}{{\mathfrak m}}

\nc{\gx}{\g^0_{\vec x}}

\nc{\hx}{\He^0_{\vec x}}
\nc{\tensorpi}{\pi_{\nu_1,\ldots,\nu_N}^\kappa}
\nc{\Phizw}{\Phi_{\vec w}({\vec z})}
\nc{\Pro}{{\mathbb P}}

\nc{\De}{D}

\nc{\us}{\underset}

\nc{\Ll}{\mc L}
\nc{\dR}{\on{dR}}

\nc{\T}{{\mc T}}

\nc{\Xn}{\overset{\circ}X{}^n} \nc{\Dn}{\overset{\circ}D{}^n}
\nc{\Dxn}{\overset{\circ}D{}^n_x} \nc{\varphitil}{\wt{\varphi}}

\nc{\lf}{{\mf l}}
\nc{\Wir}{\on{Vir}}

\nc{\bfgn}{(g_1, \ldots, g_n)}
\nc{\bfzn}{(z_1, \ldots, z_n)}

\def\Z{{\mathbb Z}}

\def\C{{\mathbb C}}

\def\wt{{\rm wt}}
\def\End{{\rm End}}

\def\ch{{\rm  ch}}

\def\1{{\bf 1}}

\def\l{{\lambda}}
\def\<{\langle}
\def\>{\rangle}

\def\Res{{\rm Res}}

\def\a{\alpha}
\def\b{\beta}
\def\M{{\mathcal M}}

\def\Hom{{\rm Hom}}

\theoremstyle{remark}

\theoremstyle{definition}

\newcommand{\cC}{{\mathcal C}}

\newcommand{\comment}[1]{}

\begin{document}
\title[K-theory cohomology of associative algebra twisted bundles] 
{K-theory cohomology of associative algebra twisted bundles}  
                                
\author{A. Zuevsky} 
\address{Institute of Mathematics \\ Czech Academy of Sciences\\ Zitna 25, Prague\\ Czech Republic}

\email{zuevsky@yahoo.com}

\begin{abstract}
We introduce and study a $K$-theory of twisted bundles for 
associative algebras $A(\mathfrak g)$ of formal series with an infinite-Lie algebra coefficients 
over arbitrary compact topological spaces.  
Fibers of such bundles are given by elements of algebraic completion of 
the space of all formal series 
in complex parameters,  
sections are provided by rational functions with 
prescribed analytic properties. 
In this paper we introduce and study K-groups $K(A(\mathfrak g), X)$ 
of twisted $A(\mathfrak g)$-bundles as equivalence classes $[\E]$ of $A(\mathfrak g)$-bundle $\E$.  
We show that for any twisted $A(\mathfrak g)$-bundle 
$\E$ there exist another bundle $\widetilde{\E}$
such that an element of $K(A(\mathfrak g), X)$ for $\E$ 
can be represented in the form $[\E]/[\widetilde{\E}]$.  
The group $K(A(\mathfrak g), X)$ homomorphism properties 
with respect to tensor product, and  
splitting properties with respect to reductions of $X$ into base points.  
We determine also cohomology of cells of K-groups 
for the factor $X/Y$ of two compact spaces $X$ and $Y$. 

\end{abstract}

\keywords{Associative algebras, fiber bundles, rational functions with prescribed properties, K-groups}

\vskip12pt  

\maketitle
\section{Data availability statement}
The author confirms  that: 

1.) All data generated or analysed during this study are included in this published article. 

2.)   Data sharing not applicable to this article as no datasets were generated or analysed during the current study.
\section{Introduction}
The equivalence classes of bundles \cite{A, DLMZ, H}
 associated to various algebraic structures defined 
on topological spaces 
allow to use combinations of algebraic and topological properties of non-commutative 
objects in terms of abelian groups. 
 It is natural to consider bundles using modules of associative algebras.  
Such bundles are important for computations in cohomology theory on smooth manifolds, 
 \cite{W, Wi, Wag}.  
 In \cite{Zu} we introduced and studied fiber twisted bundles related to 
modules of associative algebras for 
infinite-dimensional Lie algebras 
 in the form of group of automorphisms torsors originating from 
local geometry. 
 Our original motivation for
this work was to understand continuous cohomology \cite{Bott, BS, Fei, Fuks, GKF, PT, Wag}
 of non-commutatieve structrues over 
compact topological spaces. 
In particular \cite{BS},  
one hopes to relate cohomology of infinite-dimensional Lie algebras-valued series
 considered 
on complex manifolds with fiber bundles on auxiliary topological spaces \cite{PT}.  
Let $\mathfrak g$ be an infinite-dimensional Lie algebra \cite{K}. 
Starting from algebraic completion $G_{(z_1, \ldots, z_n)}$ of the space of $\mathfrak g$-valued series 
in a few formal complex parameters $(z_1, \ldots, z_n)$,    
we introduce the category $\Oo_{A(\mathfrak g)}$ of associative algebra modules for the 
associative algebra $A(\mathfrak g)$ originating  
from $G_{(z_1, \ldots, z_n)}$ by means of factorization with respect to two natural multiplications \cite{Z}. 
Local parts of twisted bundles are constructed as principal bundles of 
products of ${\rm Aut}(\mathfrak g)$-modules 
and spaces of all sets of local parameters of a $X$-covering. 
 As in the untwisted case \cite{DLMZ}, this result is crucial in defining $A(\mathfrak g)$ 
 K-groups and studying the cohomology their properties. 
In \cite{DLMZ} they explored the vertex operator algebra approach to K-theory of 
compact topological spaces. 
Vertex operator algebra $V$ bundles and associative algebra bundles related to $V$
gave rise to a series of exact sequences ot K-groups for a compact topological space. 
An alternative approach to vertex operator algebra bundles was given in \cite{BZF}. 

In this paper we introduce and determine properties of 
K-groups $K(A(\mathfrak g), X)$ defined for 
twisted $A(\mathfrak g)$-bundles for the associative algebras $A(\mathfrak g)$
for ${\mathfrak g}$ on compact topological space $X$. 
In order to formulate the K-theory, 
 we use the axiomatics of prescribed rational functions. 
We show, in particular, that all elements of $K(A(\mathfrak g), X)$ for a 
twisted $A(\mathfrak g)$-bundle $\E$ 
can be represented in the form $[\E]/[\widetilde{\E}]$ 
with another bundle $\widetilde{\E}$.  
The group $K(A(\mathfrak g), X)$ exhibits natural homomorphism properties 
with respect to tensor product of associative algebras, and possesses 
splitting properties with respect to reductions of $X$ into a point. 
We study also cohomology of cells of K-groups 
for the factor $X/Y$ of two compact spaces $X$ and $Y$. 
The cells as short exact sequences of K-groups of $X/Y$, $X$ and $Y$. 
The cohomology of K-cells is determined explicitly. 
K-groups of twisted associative algebra bundles possess non-vanishing 
cohomology in contrast to K-groups for vertex operator algebras. 
Studies of K-groups of bundles considered in this paper find their applications 
in conformal field theory \cite{BZF, TUY}, deformation theory \cite{GerSch, Kod, Ma}, 
 vertex algebras \cite{FHL, H1}, and algebraic topology \cite{Bott, Fei}. 
\section{Prescribed rational functions originating from matrix elements}   
In this Section the space of prescribed rational functions is defined as 
functions with certain analytical and symmetry properties \cite{H1}.      
They depend on an infinite number of non-commutative parameters.  
Let $X_{(\alpha)}=\left\{ X_\alpha, \alpha \in \Z_{>0} \right\}$  
 be an open covering of a compact topological space $X$ 
which gives a local trivialization of the $A(\mathfrak g)$ fiber bundle.    
 Let $\mathfrak g$, be an infinite-dimensional Lie algebra. 
Denote by $G=G_{(z_1, \ldots, z_n)}$ 
be the graded (with respect to a grading operator $K_G$)   
algebraic completion of the space of all formal series 
in each of complex formal parameters $(z_1, \ldots, z_n)$ individually, with expansion coefficients 
as elements $g \in G$,  
and satisfying certain properties described below.   
 We denote by $(x_1, \ldots, x_n)=(g_1, z_1; \ldots; g_n, z_n)$     
for $(g_1, \ldots, g_n) \in G^{\otimes n}$. 
It is assumed that on $G$ 
 there exists a non-degenerate bilinear pairing $(. , .)$. 
  $G'=\coprod_{\l \in \C} G_{\l}^*$ denotes 
the gaded dual for $G=\bigoplus_{\lambda \in{\C}}G_{(\lambda)}$ 
 with respect to $(.,.)$.
For a fixed $\theta \in G^*$,    
and varying $(x_1, \ldots, x_n)$ 
 consider the matrix elements 
\begin{equation}
\label{toppa}
F(x_1, \ldots, x_n) = \left( \theta, f (x_1, \ldots, x_n) \right),    
\end{equation}
where $F(x_1, \ldots, x_n) \in \C(z_1, \ldots, z_n)$ 
 depends implicitly on $(g_1, \ldots, g_n)$.   
In this paper we consider 
 meromorphic functions 
 of 
 defined on a compact topological space 
 which are extendable to rational functions  $R(f(z_1, \ldots, z_n))$ 
 on larger domains of several complex formal parameters $(z_1, \ldots, z_n)$.    
Denote by $F_n \mathbb C$ the  
configuration space of $n \ge 1$ ordered coordinates in $\mathbb C^n$, 
 $F_n\mathbb C=\{(z_1, \ldots, z_n) \in \mathbb C^n\;|\; z_i \ne z_j, i\ne j\}$. 
In order to work with $G$-elements $(g_1, \ldots, g_n)$ objects on $X$,  
 we consider converging rational  
functions $f(x_1, \ldots, x_n)\in G$ of $(z_1, \ldots, z_n) \in F_n \mathbb C$. 
For an arbitrary fixed $\theta \in G^*$,      
we call 
 a map linear in $(g_1, \ldots, g_n)$ and $(z_1, \ldots, z_n)$, 
$F: (x_1, \ldots, x_n) 
\mapsto   
    R(  (\theta, f(x_1, \ldots, x_n ) )$,  
  a rational function in $(z_1, \ldots, z_n)$.   
The poses are only allowed at at 
$z_i=z_j$, $i\ne j$.  
We define left action of the permutation group $S_n$ on $F(z_1, \ldots, z_n)$ 
by
$\nc{\bfzq}{(z_1, \ldots, z_n)}
\sigma(F)(x_1, \ldots, x_n)=F\left(g_1, z_{\sigma(1)};  \ldots;  g_n, z_{\sigma(n)} \right)$.  

We assume \cite{H1} that 
$G_{ (\l) }=\{w\in G | K_0 w=\l w\, \; \l=\wt(w)\}$,  
such that 
$G_{(\l)}=0$ when the real part of $\alpha$ is sufficiently negative, 
 Moreover we require that
$\dim G_{(\l)} < \infty$, i.e., 
 it is finite, and for fixed $\l$, $G_{(n+\l)}=0$,  for all 
small enough integers $n$. 
 An admissible $G$ 
is a $\C$-graded vector space with $G_{(0)} \ne 0$.   
which satisfies the following conditions.
 Assume that 
 $G$ is equipped with a map    
$\omega_g:  G[(z_1, \ldots, z_n)  \to   G[[(z_1, \ldots, z_n), (z_1^{-1}, \ldots, z^{-1}_n)]]$,   
 $g \mapsto \omega_g(z_1, \ldots, z_n)\equiv \sum_{l\in \C} g_l \; z^l$.     
For each element $g\in G$, and $(z_1, \ldots, z_n) \in \C^{\otimes n}$
 let us associate a formal series  
$\omega_g(x) = \omega_{g}(z_1, \ldots, z_n)=  
\sum\limits_{(s_1, \ldots, s_n) \in \C^{\otimes n} }  g_{(s_1, \ldots, s_n)}
 \; z^{s_1}_1 \ldots z^{s_n}_n $.  
For $g\in G$,    
$w\in $, $n \in \C$,   
$g_n w=0$, $n \gg 0$, 
 $\omega_{\bf 1}(z_1, \ldots, z_n)={\rm Id}$, 
 For $g \in G$, $\omega_g(z_1, \ldots, z_n) w$ contains only finitely many negative 
power terms, that is, $\omega_g(z_1, \ldots, z_n)w\in G$.  
The locality and associativity properties are assumed for 
matrix elements \eqref{toppa}, 
 for $g_1$, $g_2 \in \mathfrak g$, $w \in G$, 
$\theta \in G'$,   
 the series   
$\left(\theta, \omega_{g_1}(z_1) \; \omega_{g_2}(z_2)w \right)$,  
 $\left(\theta, \omega_{g_2}(z_2) \; \omega_{g_1}(z_1) w \right)$,  
$\left( \theta, \omega_{ \omega_{g_1}(z_1-z_2)  g_2 } (z_2) w \right)$,   
are absolutely convergent
in the regions $|z_1|>|z_2|>0$, $|z_2|>|z_1|>0$,
$|z_2|>|z_1-z_2|>0$, 
respectively, to a common rational function 
in $z_1$ and $z_2$. 
The poles are only allowed at $z_1=0=z_2$, $z_1=z_2$.
 If $g$ is homogeneous then 
$g_m G_{(n)}\subset G_{(\wt u-m-1+n)}$.   
For a subgroup  ${\mathfrak G} \subset {\rm Aut} \; G$,   
 $\mathfrak G$ acts on $G$ as automorphisms if 
$g \; \omega_h(z_1, \ldots, z_n) \; g^{-1}=\omega_{gh}(z_1, \ldots, z_n)$,  
for all $g$, $h \in {\mathfrak G}$.    
 For an admissible $G$ 
the operator $K_G$ satisfies the derivation property
$\omega_{K_G g}(z_1, \ldots, z_n)=\frac{d}{dz}\omega_g(z_1, \ldots, z_n)$. 
An admissible $\C$-graded $G =\bigoplus_{\lambda \in{\C}} G_{(\lambda)}$,  
 $G_{(\l)}=\{w\in G|K_0 w=\l w\}$ 
is called an ordinary. 
 We require that
 $\dim G_{(\l)}$ is finite and for fixed $\l$,  $G_{(s+\l)}=0$, 
for all small enough integers $s$. 
Let us assume that $G_{(0)}=\C\1$.
 $G$ is called rational if 
 it is a direct sum of irreducible admissible $G^i$. 
From \cite{DLM2, Z} we know that, 
all of finite number (up to isomorphisms) irreducible admissible $G^i$ of  
 a rational $G$, 
are ordinary modules. 
 For each $\l\in\C$ denote $D(\l)=\{\l+n|0\leq n\in\Z\}$. 
 The category $\Oo_\mathfrak g$ of 
ordinary $G$ is 
 that for there exist 
finitely many complex weights $(\l_1,...,\l_p)$ such that
with $P(G)=\{\l\in\C |G_{\l}\ne 0\}$, 
$P(G) \subset \bigcup_{i=1}^p D(\l_i)$.  
 Any irreducible module is in $\Oo_\mathfrak g$. 
If $G$ is rational then $\Oo_{\mathfrak g}$ is exactly the category of ordinary modules.   
For two $\mathfrak g$ and $\mathfrak g'$  
the functor $\Oo_{\mathfrak g} \times \Oo_{\mathfrak g'} \to \Oo_{\mathfrak g\otimes \mathfrak g'}$
 such that  $G^1 \times G^2 \to G^1 \otimes G^2$.   
Next recall the notion of a contragredient module \cite{FHL}. 
We denote the natural bilinear pairing on $G'\times G$
 by $(w',w)$ for $w'\in G'$ and $w \in G$.  
 It is called invariant if 
$(w', \omega_g(z_1, \ldots, z_n) w) =(\omega_{g'} w', w)$,   
where $g'=e^{z^3 \partial_z}(-z^{-2})^{K_0}z^{-1}g$, 
for $\omega_{g'}w'_i\in G'$, and $g\in G$.
 Then \cite{FHL, DLMZ} $G'$ with $\omega_g G)$  
is also a module.
 It is irreducible
 if and only if $G$ is irreducible. 
 For any $G$,  
$G' \oplus G$ has a natural
non-degenerate symmetric invariant bilinear pairing defined by
$(u+u',w+w')=(u,w')+(w,u')$
for any $u$, $w\in G$ and $u'$, $w'\in G'$. 
In particular, any
  $G$ can be embedded into a module with a nondegenerate
symmetric invariant bilinear pairing.

Let $(z_1, \ldots, z_n) \in F_n \C$.   
Denote by $T_G$ the translation operator \cite{H1}. 
Denote by $(T_G)_i$ the operator acting on the $i$-th entry.
 We then define the action of partial derivatives on an element $F(x_1, \ldots, x_n)$  
\begin{eqnarray}
\label{cond1}
\partial_{z_i} F(x_1, \ldots, x_n)  &=& F((T_G)_i \; (x_1, \ldots, x_n)),   
\nn
\sum\limits_{i \ge 1} \partial_{z_i}  F(x_1, \ldots, x_n)  
&=&  T_G F(x_1, \ldots, x_n),   
\end{eqnarray}
and call it $T_G$-derivative property. 
For $z \in \C$,  let 
\begin{eqnarray}
\label{ldir1}
 e^{zT_G} F (x_1, \ldots, x_n)   
 = F(g_1, z_1+z ; \ldots;  g_n, z_n+z).   
\end{eqnarray}
 Let us denote by 
${\rm Ins}_i(A)$ the operator of multiplication by $A \in \C$ at the $i$-th position. 
Then we assume that both sides of the expression 
$F\left((g_1, \ldots, g_n), {\rm Ins}_i(z_1, \ldots, z_n) \; (z_1, \ldots, z_n) \right)=   
F\left( {\rm Ins}_i (e^{zT_G}) \; (x_1, \ldots, x_n)\right)$,     
are absolutely convergent
on the open disk $|z|<\min_{i\ne j}\{|z_{i}-z_{j}|\}$, and 
 equal as power series expansions in $z$. 
A rational function has $K_G$-property   
if for $z\in \C^{\times}$ satisfies 
$(zz_1, \ldots, zz_n) \in F_n\C$,  
\begin{eqnarray}
\label{loconj}
z^{K_G } F (x_1, \ldots, x_n) =  
 F \left(z^{K_G} (g_1, zz_1; \ldots; g_n, zz_n) \right). 
\end{eqnarray}

Now we recall the definition of rational functions with prescribed analytical behavior
on a domain of $X$.     
We denote by $P_{k}: G \to G_{(k)}$, $k \in \C$,      
the projection of $G$ on $G_{(k)}$.
  Following \cite{H1}, we formulate the following definition. 
For $i$, $j=1, \dots, (l+k)n$, $k \ge 0$, $i\ne j$, $ 1 \le l', l'' \le n$, 
 let $(l_1, \ldots, l_n)$ be a partition of $(l+ k)n     
=\sum\limits_{i \ge 1} l_i$, and $k_i=l_{1}+\cdots +l_{i-1}$. 
For $\zeta_i \in \C$,  
define 
$h_i  =F ( \omega_{  g_{k_1+l_1} }( z_{k_1+l_1}  - \zeta_1) \ldots 
\omega_{  g_{k_l+l_l} }( z_{k_l+l_l}  - \zeta_l))$, 
for $i=1, \dots, l$.
A rational function $F$ is a rational function with prescribed analytical behavior 
if it satisfies properties \eqref{cond1}--\eqref{loconj}. 
In addition to that, it is assumed that  
the function   
 $\sum\limits_{ (r_1, \ldots, r_l) \in \Z^l}  
F(  P_{r_1} h_1, \zeta_1; \ldots;  P_{r_l} h_l, \zeta_l )$,     
is absolutely convergent to an analytically extension 
in $(z_1, \ldots, z_n)_{l+k}$ 
in the domains 
$|z_{k_i+p} -\zeta_i| 
+ |z_{k_j+q}-\zeta_j|< |\zeta_i -\zeta_j|$,    
for $i$, $j=1, \dots, k$, $i\ne j$, and for $p=1, 
\dots$,  $l_i$, $q=1$, $\dots$, $l_j$.  
The convergence and analytic extention do not depend on complex parameters $(\zeta)_{l}$. 
On the diagonal of $(z_1, \ldots, z_n)_{l+k}$   
 the order of poles is bounded from above by escribed ppositive numbers 
$\beta(g_{l',i}, g_{l'', j})$.   
 For $(g_1, \ldots, g_{l+k}) \in G^{\otimes (l+k)}$,
$z_i\ne z_j$, $i\ne j$ 
$|z_i|>|z_s|>0$, for $i=1, \dots, k$, 
$s=k+1, \dots, l+k$ the sum 
$\sum_{q\in \C}$  
$F(  \omega_{g_1}(x_1) \ldots  \omega_{g_k}(x_k)  
 P_q ( \omega_{g_{1+k}}(z_{1+k}) \ldots \omega_{g_{l+k}}(x_{l+k}) ) )$,     
is absolutely convergent and 
 analytically extendable to a
rational function 
in variables $(z_1, \ldots, z_n)_{l+k}$. 
The order of pole that is allowed at 
$z_i=z_j$ is bounded from above by the numbers 
$\beta(g_{l', i}, g_{l'', j})$. 

For $m \ge 1$ and $1\le p \le m-1$, 
 let $J_{m; p}$ be the set of elements of the group 
$S_m$ which preserve the order of the first $p$ numbers and the order of the last 
$m-p$ numbers, i.e., 
$J_{m, p}=\{\sigma\in S_m\;|\;\sigma(1)<\cdots <\sigma(p),\;
\sigma(p+1)<\cdots <\sigma(m)\}$.
Denote by $J_{m; p}^{-1}=\{\sigma\;|\; \sigma\in J_{m; p}\}$.  
In what follows, we apply the condition 
\begin{equation}
\label{shushu}
\sum_{\sigma\in J_{n; p}^{-1}}(-1)^{|\sigma|}  
\sigma( 
F (g_{\sigma(1)}, z_1; \ldots; g_{\sigma(n)}, z_n))=0.  
\end{equation}
on certain rational functions. 
 The space $\Theta \left(n, k, G_{(z_1, \ldots, z_n)}, U\right)$ 
of $n$ formal complex parameters 
  matrix elements $F(x_1, \ldots, x_n)$ 
  as the space of restricted     
 rational functions 
 with prescribed analytical behavior 
on a $F_n\C$-domain $U \subset X$,  and   
satisfying 
 $T_G$- and $K_G$-properties \eqref{cond1}--\eqref{loconj}, 
 \eqref{shushu}.  
\section{The associative algebra $A(\mathfrak g)$}  
In this Section we remind  \cite{DLM2, Z} the definition and properties of the associative algebra  
$A(\mathfrak g)$ for $\mathfrak g$.    
  For any homogeneous vectors
$h$, $\widetilde{h} \in G$,    
one defines the bilinear extention to $G \times G$ of 
the multiplications   
$h *_\kappa \widetilde{h}= \Res_{z} \left((1+z)^{ \wt(h) }    
 \sum\limits_{l \in \C}h_n z^{l-\kappa} \right).\widetilde {h}$, 
for $\kappa=1$, $2$.    
Here $\Res_z$ denotes the coefficient in front of $z^{-1}$. 
For $h$, $\widetilde{h} \in G$,   
 define $A(\mathfrak g)=G_{(z_1, \ldots, z_n)}/ (  {\rm span}  ( h*_2 \widetilde{h} ) )^\theta$, 
 $\theta=0$, $1$. 
For $\theta=0$ we get back to $G_{(z_1, \ldots, z_n)}$ with associativity property $G \times G$
and expressed via matrix elements. 
 For $\theta=1$ we obtain an associative algebra associated with ordinary associativity.  
In \cite{Z, DLMZ} we find the following   
\begin{theorem}
\label{P3.10}
 The bilinear operation $*_1$ turns $A(\mathfrak g)$ into an associative
algebra with the linear map 
$\phi:  g \mapsto \exp \left( -z^{2} \partial_z \right) (-1)^{-z \partial_z} g$, 
inducing an anti-involution $\nu$ on $A(\mathfrak g)$.
\end{theorem}
In what follows, let us denote by $W=W_{(z_1, \ldots, z_n)} \subset G_{(z_1, \ldots, z_n)}$ 
 an $A(\mathfrak g)$-module.  
 The space of lowest weight vectors of $G$ is then defined as
$L(W)=\{g \in G,  w \in W|g_{ \wt(h) + m} w=0, h\in W, m \geq 0\}$.
With $W=\bigoplus_{\l\in\C} W_{(\l)}$,   
  each homogeneous subspace $L(W)_{(\l)}= L(W)\cap W_{(\l)}$  
of the natural grading $L(W)=\bigoplus_{\l\in\C} L(W)_{(\l)}$  
is finite dimensional \cite{DLMZ, Z}. 
It is easy to see the following 
\begin{lemma}\label{l6.a} 
Let $W$, $\widetilde{W}$ be two $A(\mathfrak g)$-modules
 with an $A(\mathfrak g)$-module homomorphism 
$\varphi:  W\to \widetilde{W}$.  
 Then $\varphi(L(W)) \subset L(\widetilde{W})$.  
In particular, if $\varphi$ is an isomorphism then 
$\varphi(L(W))=L(\widetilde{W})$.
\end{lemma}
An associative algebra $A(\mathfrak g)$ is called semisimple if it is  
  a direct sum of full matrix algebras.  
 Let $A(\mathfrak g )$ and $A(\widetilde {\mathfrak g} )$ be two associative algebras with
anti-involutions $\nu_{A(\mathfrak g)}$ and $\nu_{A(\widetilde {\mathfrak g} )}$ 
respectively. Then one has 
\begin{lemma}
 $A(\mathfrak g )\otimes_\C A(\widetilde {\mathfrak g})$
 is an associative algebra with anti-involution
$\nu_{A(\mathfrak g)}\otimes \nu_{A(\widetilde {\mathfrak g})}$.  
\end{lemma}
  We denote by $W'$ the dual space to $W$ with respect to the form $(. , .)$.  
 The following lemma is obvious.   
\begin{lemma}
\label{l6.1} $W'$ is an $A(\mathfrak g)$-module such that 
 $(a \; m',m)=(m',\nu(a) \; m)$,  
 for $a\in A(\mathfrak g)$, $m'\in W'$, $m\in W$, and 
 $\nu$ is an anti-involution.  
\end{lemma}
For $w_i\in W$, $i=1$, $2$, and  
$a\in A(\mathfrak g)$, 
a form $(.,. )$ defined on $W$ is called 
invariant if $(a \; w_1, w_2) = (w_1,\nu(a) \;w_2)$. 
The category $\Oo_{A(\mathfrak g)}$ consists of $A(\mathfrak g)$-modules $W$ 
with $\l_i \in\C$, $1 \le i \le p$,  
such that 
$W=\bigoplus_{i=1, \atop n\geq 0}^p  W_{(\l_i)}$,   
 is a direct sum of finite dimensional $A(\mathfrak g)$-modules, and 
$\Hom_{A(\mathfrak g)}\left(W_{(\l)}, W_{(\mu)}\right)=0$,  
 if  $\mu \ne \l$. 
From \cite{DLMZ} we have 
\begin{theorem}
\label{P3.1}
For homogeneous $g\in G$ extended linearly to all $G$ and $W_{(0)} \ne 0$,  
 $L(W)$ is a left $A(\mathfrak g)$-module, and 
 the linear map $g \mapsto g_{(\wt(g)-1)}|L(W)$:   
$W \rightarrow \End (L(W))$,  
 induces a homomorphism from $W$ to $\End (L(W))$. 
For all $\l\in \C$,  
 $L(W)_{(\l)}$ is an finite-dimensional $A(\mathfrak g)$-module. 
\end{theorem}
\section{Twisted $A(\mathfrak g)$-bundles}
In this Section we recall \cite{Zu} the construction of      
 associative algebra $A(\mathfrak g)$ twisted bundles.      
Here we show how to organize elements of the space 
$\Theta\left(n, k, W_{ (z_1, \ldots, z_n), \l }, X_{\a} \right)$ of 
prescribed rational functions 
into sections of a twisted $A(\mathfrak g)$-bundle on $X$.  
Let $\mathcal H$ be a subgroup of the group ${\rm Aut}_{(z_1, \ldots, z_n)}\; \Oo_X$  
of $n$ independent formal parameters $(z_1, \ldots, z_n)$ automorphisms on $X$. 
A non-empty set $\mathcal X$ is called a group $\mathfrak H$-torsor \cite{BZF}  
if it is endowed with a simply transitive right action of $\mathfrak H$. 
This means that for $\xi$, $\widetilde \xi \in \mathcal X$, there exists a unique 
$h \in \mathfrak H$ such that 
$\xi \cdot h = \widetilde \xi$, 
where for $h$, $\widetilde{h} \in \mathfrak H$ the right action is given by 
$\xi \cdot (h \cdot \widetilde{h}) = (\xi \cdot  h) \cdot \widetilde{h}$. 
This construction allows us to identify $\mathcal X$ with $\mathfrak H$ by sending 
 $\xi  \cdot h$ to $h$.  

Similar to \cite{BZF}, one sees that
 certain subspaces $W \subset G_{(z_1, \ldots, z_n)}$ are $\mathcal H$-modules. 
For $\mathcal H$-torsor $X_{\a}$ of its module $W$ and  $X_\a$, 
let us associate the $X_{\a}$-twist of $W$  
$\E_{X_\a} = W \; 
{ {}_\bigtimes \atop      {}^{  \mathcal H      }     }
X_{\a}   
=  W  \times  X_\a/  \left\{ (w, a \cdot \xi) \sim  (aw, \xi) \right\}$,
for $\xi \in X_\a$, $a \in \mathcal H$, and $w \in W$.    
The isomorphisms  
$i_{(z_1, \ldots, z_n), X_\a }: W  \; \widetilde{\to} \;  \E_{X_\a}$  of $W$ 
define a representation  of $\mathcal H$ on $W$. 
  Then $\E_{X_\a}$ is canonically identified with the twist of $W$ by the 
 $\mathcal H$-torsor of $X_\a$.    
 Elements of $\Theta\left(n, k, W_{ (z_1, \ldots, z_n) }, X_\a\right)$ form   
 sections $F(x_1, \ldots, x_n)$.   
The principal bundle for the group 
$\mathcal H$ on $X$ provides the construction 
of local parts of a twisted $A(\mathfrak g)$-bundle. 
Let ${\it Aut}_{X_\a}$ be the space of all sets of local formal parameters on $X_\a$.   
We define the $\mathcal H$-twist of $W$  
$\E_{(z_1, \ldots, z_n)}=  W  \;  { {}_\bigtimes \atop      {}^{  \mathcal H      }     }
  \;{\it Aut}_{X_\a}$.      

Let us assume that a $\C$-grading on $W$ is induced 
by a $K_0$-graidng defined on $G$  
 Let ${\mathfrak G} \subset {\rm Aut}(G)$ be a $W$-grading preserving 
subgroup.   
 Denote by $\Oo_{\mathfrak G, A(\mathfrak g)}$ a 
subcategory of $\Oo_{A(\mathfrak g)}$ consisting of $A(\mathfrak g)$-modules 
 $W$ such that  
$\mathfrak G$  
acts on $W$ as automorphisms.    
By using ideas of \cite{BZF}, 
let us define the local part $\E(W_{ (z_1, \ldots, z_n), (\lambda) })$ 
of a tweisted $A(\mathfrak g)$-bundle via 
 matrix elements 
$F(x_1, \ldots, x_n)$ that belong to 
$\Theta \left(n, k, W_{ (z_1, \ldots, z_n), (\lambda) } \right)$  for all $n$, $k \ge 0$. 
 on a finite part $\left\{ X_\alpha, \alpha \in I_0 \right\}$      
 of a covering $\left\{X_\alpha \right\}$ of $X$.         
By using the properties of prescribed rational functions
we form a principal $\mathcal H$-bundle,  which is 
  a fiber bundle $\E(W_{  (z_1, \ldots, z_n), (\l) })$ 
with the fiber space 
provided by elements $f(x_1, \ldots, x_n) \in W$, and 
defined by trivializations 
$i_{ (z_1, \ldots, z_n) }: F  (x_1, \ldots, x_n)= 
  (\theta, f(x_1, \ldots, x_n)  )   \to X_\a$,    
 with a continuous free and transitive  $F(x_1, \ldots, x_n)$-preserving right action 
 $F(x_1, \ldots, x_n) \times \mathcal H \to F(x_1, \ldots, x_n)$.    
The projection ${\it Aut}_{X_\a} \rightarrow X$ is a principal 
$\mathcal H$-bundle similar to \cite{BZF}.  
 $\mathcal H$-torsor $ {\it Aut}_{X_\a}$ is 
the fiber of this bundle over $X_\a$.   
For a finite-dimensional $\mathcal H$-module $W_{i_{(z_1, \ldots, z_n)}, \lambda}$, let  
$\E( W_{ (z_1, \ldots, z_n),  \l }) = \bigoplus_{n, k \ge 0} W_{ i_{(z_1, \ldots, z_n)}, \lambda}  \;  
{ {}_\bigtimes \atop      {}^{  \mathcal H      }     }
\; {\it Aut}_{X_\a}$,    
be the fiber bundle associated to $W_{i_{(z_1, \ldots, z_n)}, \lambda} $, ${\it Aut}_{X_\a}$, 
 with sections provided by elements of $\Theta(n, k, W, X_\a)$, for $n$, $k \ge 0$. 
On $X$ we can choose  $\left\{X_\a\right\}$  
   such that the bundle $\E(W_{ (z_1, \ldots, z_n), \l })$ over $X_\a$ is  
$X_\a \times F(x_1, \ldots, x_n)$.  
The map $\E(W_{ (z_1, \ldots, z_n), \l }): \C^n  \rightarrow X$ 
is the fiber bundle $\E(W_{ (z_1, \ldots, z_n), \l })$  
with fiber $f(x_1, \ldots, x_n)$, the total space $\C^n$  
of $\E(W_{  (z_1, \ldots, z_n), \l })$, and $X$ is its base space.  
 For every $X_\a$ of $X$ 
 $i_{ (z_1, \ldots, z_n) }^{-1}$ is homeomorphic to $X_a \times \C^n$. 
For
$ f( (x_1, \ldots, x_n) ): 
 i_{   (z_1, \ldots, z_n)}^{-1} \rightarrow  X_\a  \times \C^n$,  
that 
$\mathcal P \circ  f (  (x_1, \ldots, x_n)  )      
= i_{   (z_1, \ldots, z_n) } \circ 
 {   i^{-1}_{ (z_1, \ldots, z_n) }  }
  \left( X_\a   \right)  $,  
where $\mathcal P$ is the projection map on $ X_\a$.    

A twisted fiber bundle $\E$  over $X$ associated  to 
$A(\mathfrak g)$-module with the fiber   
$W \in \Oo_{\mathfrak G, A(\mathfrak g)}$ and 
 $\Theta\left(n, k, W, X \right)$, $n$, $k \ge 0$-valued sections  
  is a direct sum of vector bundles  
$\E=\bigoplus_{\l \in \C} \E(W_{ (z_1, \ldots, z_n), \lambda} )$,  
  such that all transition functions are 
$A(\mathfrak g)$-module isomorphisms. 
It includes a family of continuous 
isomorphisms 
 $H_\a =\left\{ H_{\a, \l}, \l\in \C\right\}$,  
 $H_{\a, \lambda} : \E(W_{(\lambda}) )|_{X_\a } \to W_{ i_{(z_1, \ldots, z_n)}, \l} \; 
{ {}_\bigtimes \atop      {}^{  \mathfrak G     }     } \; 
 {\it Aut}_{X_\a}$,     
 of fiber bundles such that for transition functions 
$g_{  \a  \b, \l}  
=  H_{\a, \lambda} *_2 H_{\b, \lambda}^{-1}$,   
for all $\lambda \in \C$. 
 Then 
$g_{\a \b}(x)= \left(g_{\a\b, \l}(\xi) \right): W \to W$,     
 are $A(\mathfrak g)$-module isomorphisms   
for any $\xi \in  X_\a \bigcap  X_\b$ where   
 the transition functions 
$g_{\a, \b}(x)$ are $G$-valued.     

Now we recall 
 properties of twisted $A(\mathfrak g)$-bundles \cite{Zu}.  
 For $A(\mathfrak g)=\C$, the twisted $A(\mathfrak g)$-bundle  
is a classical complex vector bundle over $X$.  
  We have  
\begin{lemma}
\label{tufta}
For $\E_1$, $\E_2$ 
  $A(\mathfrak g_1)$- and $A(\mathfrak g_2)$-bundles 
$\E_1\otimes \E_2$ is 
$A(\mathfrak g_1) \otimes_{\C} A(\mathfrak g_2)$-bundle over $X$. 
\end{lemma} 
For $\E$, let us introduce $\E'=\oplus_{\l\in \C}(\E(W_{(\lambda)} ))^*$ which is, due to 
properties of the non-degenerate bilinear pairing $(., .)$, 
 is also a   $A(\mathfrak g)$-bundle. 
 For two bundles $\E$ and $\E'$ on $X$, 
a map
$\eta: \E \to \E'$, 
 is called a bundle morphism    
if there exist a 
family of continuous morphisms of fiber bundles 
$\eta_\l:  \E( W_{(\lambda)} )  \to \E'( W_{(\lambda)} )$, 
 such that with $\eta = \left(   \eta_\l  \right)$, for all $\l\in \C$, and 
 $\eta_\l: \E \to \E'$,  
is an $A(\mathfrak g)$-module morphism.    
From \cite{DLMZ, Zu} we find 
\begin{lemma}
\label{l3.1}
$\E \bigoplus \E'$ is a twisted $A(\mathfrak g)$-bundle,  
 endowed with non-degenerate symmetric
invariant bilinear pairing  
$\left(g_{\a\b}^*(\xi)\theta, g_{\a\b}(\xi) u \right)=\left(\theta, u \right)$,   
 independent of $g_{\a\b}$ for all $\a$, $\b \in I$, $\xi \in X_\a\cap X_\b$,
$\xi \in G$, $\theta \in G'$, and    
induced by the natural bilinear pairing on
$G \oplus  G'$.   
\end{lemma}
A twisted $A(\mathfrak g)$-bundle $\E$ is called trivial 
if there exists an $A(\mathfrak g)$-bundle 
isomorphism
 $\varphi: \E \to W \times X$,   
here $W \times X$ is the 
$A(\mathfrak g)$-bundle on $X$ with $W$ as fibers.  
 For any $A(\mathfrak g)$-module $M \in \Oo_{A(\mathfrak g)}$ we 
denote the trivial $A(\mathfrak g)$-bundle on $X$ by $\M$. 
Let us now understand how the notion of a trivial bundle 
is determined by the definition of a twisted $A(\mathfrak g)$-bundle. 
Let $W \in \Oo_{A(\mathfrak g)}$.   
A subgroup $\mathcal H \subset {\rm Aut}_{(z_1, \ldots, z_n)} \Oo_X$ 
determines a trivial bundle $W \times X$ 
if $\mathcal H$ satisfies the following properties. 
$W$ is a $\mathcal H$-module. 
The $\mathcal H$ twist $\E(X_\alpha)$ preserves $W$  
under ${\rm Aut}_{(z_1, \ldots, z_n)} \Oo_{X_\alpha}$-actions by isomorphisms $i_{(z_1, \ldots, z_n)}$. 
The local part of the bundle $W \times X$ is given by the principle $\mathcal H$-bundle 
with $\mathcal H$-actions preserving sections $F(x_1, \ldots, x_n) \to X_\alpha$.  
The trivial bundle $W \times X$ is the direct sum of trivial bundles 
with transition functions $g$ given by $W$-preserving
 triansition functions of $\left\{ X_\alpha \right\}$, $\alpha \in I$. 
From \cite{Zu} we have 
\begin{proposition}
\label{p3.4} 
For any twisted $A(\mathfrak g)$-bundle $\E$,   
there exists a twisted 
$A(\mathfrak g)$-bundle $\widetilde {\E}$ such that $\E \bigoplus \widetilde{\E}$ 
 is a trivial twisted $A(\mathfrak g)$-bundle 
with a $A(\mathfrak g)$-module $W$-preserving action 
of a subgroup of ${\rm Aut}_{(z_1, \ldots, z_n)} \; \Oo_X$.  
\end{proposition}
 Twisted $A(\mathfrak g)$-bundles possess  
 the following homotopy-stability property: 
\begin{proposition} 
\label{l4.1}  
For a homotopy $\tau_t: \widetilde {X}\to X$, 
 $0\leq t \leq 1$, of a compact Hausdorff space $\widetilde{X}$, 
 and a twisted $A(\mathfrak g)$-bundle $\E$ over $X$, 
 $\tau^*_0 (\E) \simeq \tau^*_1 (\E)$. 
\end{proposition} 
We then have 
\begin{proposition}
\label{p3.6} 
For a rational $A(\mathfrak g)$-module $W$ with a decomposition 
 into irreducible modules $W^i$, 
 any twisted 
$A(\mathfrak g)$-bundle over $X$ 
$\E\simeq\oplus_{i=1}^p\V(\E)^i \otimes \W^i$ 
 with trivial twisted bundles $\W^i$ 
associated to $W^i$, and 
 vector bundles $\V(\E)^i$.  
\end{proposition}
\begin{proof}
 Let $\E$ be a twisted $A(\mathfrak g)$-bundle with fiber 
 $W=\oplus_{i=1}^p M_i \otimes W^i$,  
 where 
$M_i$ is the space of multiplicity of $W^i$ in $W$. 
Then each set of transition functions 
$\left\{g_{\a\b} \right\}$
 defines a map $h_{\a\b}:  X_\a\cap  X_\b \to \bigoplus_{i=1}^p \End (M_i)$. 
 For each $i$ we define a vector
bundle $\V(\E)^i$ over $X$ with fiber 
$M_i$ and transition functions
$\left\{ h_{\a\b}|\a,\b\in I\right\}$. 
Then we have a
$A(\mathfrak g)$-bundle $\V(\E)^i\otimes \W^i$.  
\end{proof}
\section{$K\left(A(\mathfrak g), X\right)$-group  
for twisted $A(\mathfrak g)$-bundles}    
\label{s4}
In this Seciton, for an associative algebra $A(\mathfrak g)$, 
 we introduce 
 the K-group $K(A(\mathfrak g), X)$ of a 
twisted $A(\mathfrak g)$-bundle on a compact topological space $X$, 
 and study corresponding properties.  
We follow the set-ups of \cite{A, DLMZ} with certain necessary extensions.
According to the definition of a twisted $A(\mathfrak g)$-bundle $\E$,  
 the set of its equivalence classes $[\E]$   
  is an abelian semigroup with addition given by the direct sum.  
 We denote by $K(A(\mathfrak g), X)$ the abelian group generated by the set of 
equivalence classes $[\E]$ of a twisted $A(\mathfrak g)$-bundle $\E$.     
For $\mathfrak g=\C$, the group $K(A(\mathfrak g), X)$ becomes the ordinary group $K(X)$ 
 as in \cite{A}.   
Let us define by  
$\Omega_0=\left(\mathfrak g, W\subset G_{(z_1, \ldots, z_n)}, K_G, (.,.) \right.$, 
 $\left. \theta, \beta(g', g''), K_0;
 \mathcal H, i_{(z_1, \ldots, z_n)},  X_{\alpha}, \alpha \in I 
 \right)$ 
 the extented moduli space for a twisted $A(\mathfrak g)$-bundle defined on a compact topological space $X$.
Let us now describe the set of parameters the isomorphism classes of twisted $A(\mathfrak g)$-bundles 
 depend on.  
Recall that $\mathcal H \subset {\rm Aut}_{(z_1, \ldots, z_n)}\; \Oo_{X}$   
of independent formal parameters $(z_1, \ldots, z_n)$ automorphisms on $X$. 
The construction involves 
 the category of subsets $W \subset G$   
which is a $\mathcal H$-module. 
$\mathcal H$-torsors and $X_{\a}$-twists of $W$ on $X_{\a}$ 
are determined by $\mathcal H$, $Aut_{X_\alpha}$, and $W$.   
 The elements of $\Theta\left(n, k, W, X_\a\right)$     
give rise to a collection of sections $F(x_1, \ldots, x_n)$ 
as prescribed rational functions.
The space $\Theta \left(n, k, W_{ (z_1, \ldots, z_n), \lambda } \right)$  for all $n$, $k \ge 0$,
 of prescribed rational functions is fixed by assumptions of Lemma. 
The module  $W \subset \Oo(A(\mathfrak g))$ is endowed with 
a $\C$-grading generated by $K_0$.   
Note that the trivializations   
$i_(z_1, \ldots, z_n) : F  (x_1, \ldots, x_n)= 
  (\theta, f(x_1, \ldots, x_n)  )   \to X_a$ 
are chosen in such way that they preserve the space $\Theta\left(n, k, W_{ (z_1, \ldots, z_n), \lambda } \right)$. 
The choice of trivializations 
$i_{(z_1, \ldots, z_n), X_\a }: W \; \widetilde{\to} \;  \E_{X_\a}$, 
is coherent with the choice of $\left\{X_\alpha\right\}$. 
Suppose we take another system of domains $\left\{ X'_{\alpha'} \right\}$. 
One shows that the construction of the bundle does not depend on the  
choice of transition functions on the intersections $X_\alpha \cap X'_{\alpha'}$. 
Thus, the equivalence classes do not depend on the choice of a covering. 
Since, by construction, the transition functions 
 $g_{\a \b}(x)= \left(g_{\a\b, \l}(\xi) \right): W \to W$,      
 are $A(\mathfrak g)$-module isomorphisms   
for any $\xi \in  X_\a \bigcap  X_\b$, 
then the construction of the bundle is invariant of the choice of transition functions. 
The twisted $A(\mathfrak g)$-bundle $\widetilde{\E}$ in Proposition \ref{p3.4} is constructed as 
follows \cite{Zu}. 
  Define an $A(\mathfrak g)$-bundle injective and bilinear form preserving homomorphism 
  $\psi: \E \to  W^{\; \oplus \;  s} \times X$ for some $s$,  
sending $\E$ to a trivial $A(\mathfrak g)$-bundle. 
Let us take $\widetilde{\E} =\psi( \E)^{\dagger}$ with respect to the bilinear form $(.,.)$.   
As the dual to $\E$, $\widetilde{\E}$ is an $A(\mathfrak g)$-bundle on $X$. 
From Proposition \ref{p3.4} we infer 
\begin{lemma}
\label{alja}
For the fixed set of data $(\mathfrak g$, 
$W \subset G$, $K_G$-grading, 
 $(.,.)$, $\theta \in W'_{(z_1, \ldots, z_n)}$, $\beta(g', g''))$ of $\Omega_0$,  
 the elements of 
$K(A(\mathfrak g), X)$ are 
of the form $[\E]/[\widetilde \E]|_{\mathcal H.W}$ 
up to $W$-preserving action 
 $\mathcal H.W$ 
of a subgroup  $\mathcal H \subset {\rm Aut}_{(z_1, \ldots, z_n)}$\; $\Oo_{\mathcal H.X}$,  
where ${\mathcal H.X}$ denotes the $\mathcal H$ preserving section and transition functions.   
\end{lemma}
The statement of following Lemma follows from 
 Lemma \ref{tufta} and the construction of a twisted $A(\mathfrak g)$-bundle. 
\begin{lemma}
\label{l4.0} 
For two infinite-dimensional algebras $\mathfrak g$ and $\widetilde {\mathfrak g}$,  
the tensor product
of the twisted $A(\mathfrak g)$- and $A(\widetilde{\mathfrak g})$-bundles 
induces a natural group homomorphism
$K(A(\mathfrak g), X)\otimes_{\Z} K(A(\widetilde{\mathfrak g}), X) 
\to K( A(\mathfrak g) \otimes A(\widetilde{\mathfrak g}), X)$.  
\end{lemma}
According to the definition of a twisted $A(\mathfrak g)$-bundle $\E$, the set of isomorphis classes 
of $\E$ is determined in particular by $\mathcal H$-invariant modules $W$ and 
  $\mathcal H$-invariant sections $F(z_1, \ldots, z_n)$.
The axioms of prescribed rational functions $F(x_1, \ldots, x_n)$ form a functional representation 
of $G$ with additional analytic behavior properties. 
 Lemma \ref{l4.0} shows that for each $s$ elements of the K-group is the tensor product of $A(\mathfrak g)$ 
 where each element if represented by commutative elements of $F(z_1, \ldots, z_n)$.  
Therefore, from Lemma \ref{l4.0} we obtain  
\begin{lemma}
\label{c} 
With $A(\mathfrak g)^{\otimes 0}\simeq \C$, the group $K=\bigoplus\limits_{s \geq 0}$    
$K\left(A(\mathfrak g)^{\otimes s}, X\right)$   
form a commutative algebra over $K(A(\mathfrak g), X)$. 
\end{lemma}
Next we state 
\begin{lemma}
\label{l4.3} 
 Any element of $K(A(\mathfrak g), X)$ has the form 
 $[\E]/[\M]$, where $\E$ is a twisted $A(\mathfrak g)$-bundle and 
$M$ is a $A(\mathfrak g)$-module. 
 For two equivalent classes $[\E]$, $[\E'] \in K(A(\mathfrak g), X)$,  
  there exists a $A(\mathfrak g)$-module $M$ such 
that $\E \cong \E'$ up to the trivial $A(\mathfrak g)$-bundle $\M$.  
\end{lemma}
\begin{proof} 
As we have shown in Lemma \ref{alja}, for every $\E$ one constructs a bundle $\E'$
such that 
 every element of $K(A(\mathfrak g), X)$ is of 
 the form $[\E]/[\E']$ up to $A(\mathfrak g)$-module $M$-preserving action of 
a subgroup $\mathcal H \subset {\rm Aut}_{(z_1, \ldots, z_n)} \; \Oo_X$.   
By Proposition \ref{p3.4},  
for any twisted $A(\mathfrak g)$-bundle $\E$,  
there exists a twisted 
$A(\mathfrak g)$-bundle $\widetilde{\E}$ 
and a $A(\mathfrak g)$-module $M$ 
such that $\E \oplus \widetilde{\E} \cong \M$. 
Thus we have
$[\E']/[\E]=  [\E' \oplus \widetilde{\E}]/[\E \oplus \widetilde{\E}]
=[\E'\oplus \widetilde{E}]/[\M]$.
If $[\E]=[\E']$ then there exists a $A(\mathfrak g)$-bundle $\E''$ such that
$\E \oplus \E'' \cong \widetilde{\E} \oplus \E''$. 
Let $\E$ be a $A(\mathfrak g)$-bundle such that
$\E'' \oplus \E\cong \M$ for some $A(\mathfrak g)$-module $M$. 
Then we have $E\oplus \M\cong \E' \oplus \M$. 
\end{proof}
Here we describe the group $K(A(\mathfrak g), (\xi_1, \ldots, \xi_n))$ for 
$(\xi_1, \ldots, \xi_n) \in X$. 
According to the construction of a twisted $A(\mathfrak g)$-bundle,  
one has a subgroup $\mathcal H_0 \subset {\rm Aut}_{(z_{\xi_1}, \ldots, z_{\xi_n})} \Oo_{(\xi_1, \ldots, \xi_n)}$ 
over a set $(\xi_1, \ldots, \xi_n) \in X$,    
such that an $A(\mathfrak g)$-module $W$ is $\mathcal H_0$-module.   
The local part of $\E$
is defined by $\E_{(\xi_1, \ldots, \xi_n)}$ $\mathcal H_0$-twists of $W_{(z_{\xi_1}, \ldots, z_{\xi_n})}$ 
attached to $(\xi_1, \ldots, \xi_n)$  
via mappings $i_{(z_1, \ldots, z_n)}: W_{(z_{\xi_1}, \ldots, z_{\xi_n})} \to \E_{(\xi_1, \ldots, \xi_n)}$. 
An invariant $\mathcal H_0$-action on $F(x_1, \ldots, x_n)$ is assumed. 
Basically, 
$\E_{(\xi_1, \ldots, \xi_n)}=W_{(z_{\xi_1}, \ldots, z_{\xi_n})}
 {\times \atop S_n} Aut_{(\xi_1, \ldots, \xi_n)}$. 
 The equivalence classes $[\E]|_{(\xi_1, \ldots, \xi_n)}$  
of twisted $A(\mathfrak g)$-bundle 
 $\E$ over a set of points $(\xi_1, \ldots, \xi_n) \in X$.  
 are given by equivalence classes of ${\mathcal H}_0$-invariant 
${\rm Aut}_{(z_{\xi_1}, \ldots, z_{\xi_n})} \Oo_{(\xi_1, \ldots, \xi_n)}$
 modules $W_{(z_{\xi_1}, \ldots, z_{\xi_n})}$. 
If $A(\mathfrak g)=\bigoplus_{i=1}^p A(\mathfrak g)$ is semisimple 
up to isomorphisms, then $K(A(\mathfrak g), (\xi_1, \ldots, \xi_n))$ is
isomorphic to the group $\Z^p\times \ldots \times \Z^p$ with generators $[\M^1], \ldots, [\M^p]$ 
for inequivalent $A(\mathfrak g)$-modules $M^1, \ldots, M^p$. 
Generalizing results of \cite{DLMZ}, we obtain. 
\begin{lemma}
\label{l4.4} 
For a semisimple associative algebra $A(\mathfrak g)$   
$K(A(\mathfrak g),  X)= K(X) \otimes_\Z K(A(\mathfrak g), (\xi_1, \ldots, \xi_n))$.   
\end{lemma}
\begin{proof}
 For rational $G$ correspoding twisted 
 $A(\mathfrak g)$-bundle $\E$ is determined by certain vector bundles. 
 Assume that its fiber $M=\oplus_{i=1}^p E_i\otimes M^i$ 
where $\{M^1, \ldots, M^p\}$ contains all inequivalent
 irreducible $A(\mathfrak g)$-modules, 
and $E_i$ is the space of multiplicity of $M^i$ in $M$.
Similar to Proposition \ref{p3.6}, we know that 
 if $A(\mathfrak g)$ is semisimple 
 then there exist vector bundles $\V(\E)^i$ for
$i=1, \ldots, p$ such that
 $\E \simeq \bigoplus_{i=1}^p \V\left(\E\right)^i \otimes \M^i$, where 
trivial bundles $\M^i$ correspond to 
 inequivalent set of $A(\mathfrak g)$-modules $M^i$.    
According our description of a twisted $A(\mathfrak g)$-bundle over a set 
$(\xi_1, \ldots, \xi_n) \in X$, 
one has a map to equivalence classes which gives 
$K(\bigoplus_{i=1}^p A(\mathfrak g)^i, X)= \bigoplus_{i=1}^p K( A(\mathfrak g)^i, X)=  
K(\bigoplus_{i=1}^p \V\left(\E\right)^i, X) \otimes_\Z K(\M^i)=
K(X)\otimes_\Z K(A(\mathfrak g), (\xi_1, \ldots, \xi_n))$.
\end{proof}
\section{Cohomology of $K$-cells} 
\label{s4.2}
By using results of \cite{DLMZ}, 
 we first define K-groups for twisted associative algebra
 bundles on factor spaces of compact topological spaces.    
Let $\cC_c$ denote the category of compact spaces, $\cC_0$ the category
of compact spaces with distinguished basepoints. 
We define a functor $\cC_0 \times \cC_0 \to \cC_0$, 
 by associating two compact topological spaces $X$, $Y$
 to a compact space $X/Y$ with base points $(\xi_1, \ldots, \xi_n)=Y/Y$. 
In the case $Y\neq \left\{ \emptyset \right\}$,  
 $X/Y$ is the disjoint union of $X$ with points $(\xi_1, \ldots, \xi_n)$, and 
 for $X\in\cC_c$, we denote $X_0= X/\left\{\emptyset\right\}$.  
If $X$ is in $\cC_0$, we define a functor 
 ${\rm Ker} \; i_0^*(X)$
to be the kernel of
 the map $i_0^*: K(A(\mathfrak g), X) \to K(A(\mathfrak g), (\xi_1, \ldots, \xi_n ))$
 where $i_0: (\xi_1, \ldots, \xi_n) \to X$ is the  
inclusion of the basepoints.  
If $c: X \to (\xi_1, \ldots, \xi_n)$ is the collapsing map 
then $c^*$ induces a splitting  
$K(A(\mathfrak g), X)= {\rm Ker} \; i_0^*(X) \oplus K(A(\mathfrak g), (\xi_1, \ldots, \xi_n))$. 
It is clear that $K(A(\mathfrak g), X) \simeq {\rm Ker} \; i^*(X_0)$. 
Now we define $K(A(\mathfrak g), X/Y)= {\rm Ker} \; i_0^*(X/Y)$. 
 In particular,  $K(A(\mathfrak g), X)  \simeq K(X_0)$.

Now let us introduce \cite{DLMZ} the generalized smash product operator in $\cC_0$,
 for $X$, $Y\in \cC_0$. 
 We put $X \wedge Y = X \times Y/X \vee Y$ 
for $X\vee Y = X\times \{(\xi_1, \ldots, \xi_n) \}
 \bigcup \{(\xi'_1, \ldots, \xi'_n) \} \times Y$, 
 being base-points of $X$, $Y$ respectively. 
For any triple of spaces
$X$, $Y$, $Z\in \cC_0$,  
 one has a natural homeomorphism $X\wedge (Y\wedge Z)\simeq (X\wedge Y)\wedge Z$. 
Let $I(t)$, $0 \le t \le 1$ we denote the boundary 
$\partial I(t) = \{0, 1\}$.  
Then $I(t)/\partial I(t) \simeq S^1 \in \cC_0$.  
For $X\in \cC_0$ we define the operator $\mathcal S$ 
of the reduced suspension on $X$ as  
  ${\mathcal S}X = S^1\wedge X \in \cC_0$. 
Then one has the $s$-th power of iterated suspensions 
${\mathcal S}^s X \cong S^s \wedge X$. 
 For $X \in \cC_0$ and $s \geq 0$ we consider 
${\rm Ker} \; i_0^*({\mathcal S}^s X)$. 
For $X$, $Y \in \cC_c$ let us define
 $K(A(\mathfrak g), X/Y)^{-s}$ $=$  $\left({\rm Ker} \; i_0^*\right)^{-s}$ $(X/Y)= 
 {\rm Ker} \; i_0^* ({\mathcal S}^s (X/Y))$, 
 $K(A(\mathfrak g), X)^{-s}$ $ = $ $ K(A(\mathfrak g), X_0)^{-s}$ 
= ${\rm Ker} \; i_0^* ({\mathcal S}^s (X_0))$. 
 We next define the cone functor $C: \cC_c \to \cC_0$ on $X$ by 
$C X= I(t) \times X/\{0\} \times X$, and 
 identify $X$ with the
  $\{1\}\times X \subset CX$.
 One calls 
$CX/X = I(t)\times X/\partial I(t)\times X$ the unreduced suspension of $X$. 
  In what follows, we assume that $X$ is a finite CW-complex
 and  $Y\subset X$ is a CW sub-complex. 
Next, we have a modification of a Lemma from \cite{DLMZ} in our context of twisted 
$A(\mathfrak g)$-bundles: 
\begin{lemma}
\label{t5.4}
For $\xi \in {\rm Ker} \; i_0^*({\mathcal S}X)$, 
  $T(t)= 1-t$, $0 \le t \le 1$ such that   
  $(T\wedge 1) : {\mathcal S}X \to {\mathcal S}X$  
and the identity on $X\in \cC_0$.  
Then $(T\wedge 1)^* \xi = - \xi$ for $\xi$.   
\end{lemma}
\begin{proof}
  By construction of a twisted $A(\mathfrak g)$-bundle and according to Lemma \ref{l4.3},
 for any $\xi \in {\rm Ker} \; i_0^*({\mathcal S}X)$, 
 there exist a $A(\mathfrak g)$-bundle $\E$ and a $A(\mathfrak g)$-module $M$  
 such that $\xi = [\E]/[\M]|_{\mathcal H.M}$, where $\M$ is a trivial bundle associated to $M$. 
For two cones $C_1 X$ and $C_2 X$, 
 on the contructible $C_1 X \cup {C_2} X = C_1 X_{C_2} X={\mathcal S}X$,  
 the restrictions $\E|_{C X}$, $\E|_{C_2 X}$ of $\E$  
  are trivial. 
Then we can define maps 
 $\rho_i :\E|_{C_i X} \to  M \times C_i X$, $1 \le i \le 2$, such that   
 $\rho= \rho_2 \circ \rho_1^{-1}=
\rho^s: X \to {\rm Aut}_{G_{(z_1, \ldots, z_n)}} (M)$, $s \in \C$. 
According to definition of the smash product above, 
$(T\wedge 1)^*$ acts on $[\E]$ 
 $x\to \rho(x)^{-1}= (\rho^s)^{-1}$
 resulting in $\E' \in K(A(\mathfrak g), {\mathcal S}X)$.  
Since $\rho_i$ are isomorphisms of $M$, 
 from Lemma \ref{l3.1} we obtain 
that classes of isomorphisms of $\E$ and $\E'$ are the direct sum of classes of $\M$ and $\M'$ 
 in $K(A(\mathfrak g), {\mathcal S}X)$.  
\end{proof}
Now we establish the  
 cohomological properties of $K(A(\mathfrak g), X/Y)$. 
\begin{lemma}  \label{t5.1} 
For inclusions $i: Y\to X$, $j: X_0 \to X/Y$, 
 the cohomology of the cell    
${\mathcal K}_{X/Y}  
=K(A(\mathfrak g), X/Y) \stackrel{j^*}{\to} K(A(\mathfrak g), X) \stackrel{i ^*}{\to} K(A(\mathfrak g), Y)$ 
is given by $H_{K(A(\mathfrak g), X/Y)}= {\rm Ker}\;  i^*/{\rm Im}\; j^*
=
 {\rm Aut}_{(z_1, \ldots, z_n)} \Oo_{X/Y}|_{\mathcal H. W}$. 
\end{lemma}
\begin{proof}
Let $Y_0=Y/\left\{ \emptyset \right\}$. 
  The composition $i^*\circ j^*$ is induced by the composition
$j\circ i: Y_0 \to X/Y$, 
and $i^*\circ j^*=0$. 
Suppose now that $\xi \in {\rm Ker} \; i^*$.
According to Lemma \ref{l4.3}, 
 $\xi$ is represented in the form $([\E]/[\M])|_{\mathcal H.W_{(z_1, \ldots, z_n)}}$  
where $\E$ is a twisted 
$A(\mathfrak g)$-bundle over $X$ and $M$ is a $A(\mathfrak g)$-module. 
Since $i^*\xi =0$, 
it follows that $[\E]|_Y = [\M]$ in $K(A(\mathfrak g), Y)$.
 According to Lemma \ref{l4.3}, this implies that  
there exists a $A(\mathfrak g)$-module $M'|_{\mathcal H.W}$ such that 
 $(\E\oplus \M')|_{\mathcal H.W} = (\M \oplus \M')|_{\mathcal H.W}$,     
and the bundle $\M'$ is trivial. 
Now as $Y$ is a CW sub-complex of $X$, there exists an open
neighborhood $\widetilde{Y}$ of $Y$ in $X$ such that $Y$ 
satisfies the following condition with respect to $\widetilde{Y}$.  
For $0 \le t \le 1$,  
one can find a map $f(t): \widetilde{Y} \to \widetilde{Y}$  
such that $f(1)={\rm Id}_{\widetilde{Y}}$, 
$f(0)|_Y={\rm Id}_Y$, and 
$f(0)(\widetilde{Y})=Y$.
 By Lemma \ref{l4.1} 
the triviality of 
 the blow-up $(\E\oplus \M')|_{\widetilde{Y}}$ 
of $(\E\oplus \M')|_Y$ is homotopically preserved, 
 on $\widetilde{Y}$. 
 This defines a bundle $\E\oplus \M'/\alpha$ on $X/Y$, 
and an element 
$\tau = [\E\oplus \M'/\gamma]/[\M\oplus \M' ] \in {\rm Ker} \; i_0^*(X/Y) = K(A(\mathfrak g), X/Y)$, 
where $I^*: K(A(\mathfrak g), X) \to K(A(\mathfrak g), \xi)$. 
Since $\M'$ is trivial, 
$j^*(b \tau) =[\E\oplus \M'/\gamma]/[\M\oplus \M']|_{\mathcal H.W } 
=  [\E]/[\M]_{\mathcal H.W}= b \xi$, 
where $b \in {\rm Aut}_{(z_1, \ldots, z_n)} \; \Oo_{X/Y}$. 
Thus ${\rm Ker}\; i^* = \left({\rm Aut}_{(z_1, \ldots, z_n)} \; \Oo_{X/Y} \right) \; {\rm Im} \; j^*$. 
\end{proof}
\begin{lemma}
 \label{t5.2}
 For $X/Y \in \cC_c$ and $Y\in \cC_0$,
 the cohomology of 
the sequence 
$K(A(\mathfrak g), X/Y) \to {\rm Ker}\; i_0^*(X) \to {\rm Ker}\; i_0^*(Y)$,  
is given by ${\rm Aut}_{(z_{\xi_1}, \ldots, z_{\xi_n})} 
\; \Oo_{(\xi_1, \ldots, \xi_n)}|_{\mathcal H.W  }$  
up to the twist $\E_{\mathcal H.W }$. 
\end{lemma}
\begin{proof}
The statement follows from Lemma \ref{t5.1} and the decomposition isomorphisms 
 $K(A(\mathfrak g), X) \simeq {\rm Ker}\; i_0^*(X) \oplus K(A(\mathfrak g), (\xi_1, \ldots,  \xi_n))$,
 $K(A(\mathfrak g), Y)\simeq {\rm Ker}\; i_0^*(Y) \oplus K(A(\mathfrak g), (\xi'_1, \ldots, \xi'_n ))$.
Thus, we obtain the statement of the lemma with the twist inserted. 
\end{proof}
For $s \ge 0$ we next define the $(-s)$-th power of the cell of K-groups   
\[
{\mathcal K}_{X/Y}^{-s}  =  K(A(\mathfrak g), X/Y)^{-s}   
\stackrel{j^*}{\to} 
K(A(\mathfrak g), X)^{-s} \stackrel{i^*}{\to} K(A(\mathfrak g), Y)^{-s}.   
\]
For a map $\nu_1: Z \to Z/X$,  
and the isomorphism $\vartheta: K(A(\mathfrak g),  Z/X) \to K(A(\mathfrak g), Y)^{-1}$, 
let us define the operator 
$\delta^{(s)}: K(A(\mathfrak g), X/Y)^{-s} \to K(A(\mathfrak g), X/Y)^{-s-1}$,    
  as the operator $\delta = \nu_1^* \circ \vartheta^{-1}$ acting on $K(A(\mathfrak g), Y)^{-s}$.  
The main proposition of this paper is 
\begin{proposition}
\label{t5.3} 
The $s$-th cell cohomology 
 $H^s_{{\mathcal K}_{X/Y}}= {\rm Ker}\;  \delta^{(-s)}/{\rm Im}\;  \delta^{-(s+1)}$,   
of the left-infinite sequence  
\begin{equation*}
\cdots \stackrel{\delta^{(-s-2)}}{\longrightarrow}
{\mathcal K}_{X/Y}^{-s-1} \stackrel{\delta^{-(s+1)}}{\longrightarrow} 
{\mathcal K}_{X/Y}^{-s}
\stackrel{\delta^{(-s)}}{\longrightarrow} \ldots
 \stackrel{\delta^{(-1)}}{\longrightarrow} 
{\mathcal K}_{X/Y}^0 \stackrel{\delta^{(0)}}{\rightarrow}  0,  
\end{equation*}
for $s \ge 0$, 
is given by ${\rm Aut}_{(z_1, \ldots, z_n)} \Oo_{{\mathcal S}^s(X/Y)}|_{\mathcal H. W}$.  
\end{proposition}
\begin{proof} 
The proof is a modification of the proof of Theorem (5.3) of \cite{DLMZ}  
for our case. 
Let us set 
${\mathcal K}^0_0 = {\rm Ker}\; i_0^*(X/Y) \stackrel{j^*}{\to} 
 {\rm Ker}\; i_0^*(X) \stackrel{i^*}{\to}  {\rm Ker}\; i_0^*(Y)$. 
One can see that the cohomology of the main sequence is equivalent to  
the cohomology of the sequence 
\begin{equation}
 \label{sequence1}
\ldots \stackrel{\delta^{(-2)}}{\to}  {\mathcal K}_{X/Y}^{-1}
\stackrel{\delta^{(-1)}}{\to}  {\mathcal K}^0_0 \stackrel{\delta^{(0)}}{\rightarrow}  0, 
\end{equation} 
 defined by $H^0_{ {\mathcal K}_{X/Y} }={\rm Ker}\; \delta^{(-1)}/{\rm Im}\; \delta^{(0)}$
 is ${\rm Aut}_{(z_1, \ldots, z_n)}$  
$\Oo_{X/Y}|_{\mathcal H. W}$. 
 By replacing
$X/Y$ by ${\mathcal S}^s(X/Y)$ for $s \ge 1$,  
we obtain cohomology of an infinite sequence continuing (\ref{sequence1}). 
 Then by replacing
$X/Y$ by $X_0/Y_0$ where $X/Y$ is any pair in $\cC_c\times \cC_c$,
 we get the cohomology of the infinite sequence in our Proposition. 
Recall that Lemma \ref{t5.2} gives the 
cohomology of mappings inside the cell $K_0$ of 
(\ref{sequence1}). 
To determine cohomology of other parts of the sequence 
 will apply Lemma \ref{t5.2} to $Z=X_CY=X \bigcup CY$,  
$Z/X$ and $Z_CX/Z=(Z\bigcup CX)/Z$. 
Let $\nu_1:Z \to Z/X$ and $\nu_2: X \to Z$ bet two inclusions. 
 By considering $Z/X$ we get the cohomology 
$H_{K_0, Z}={\rm Ker}\; \nu_2^*/{\rm Im}\; \nu_1^*$ of the sequence 
$K(A(\mathfrak g), Z/X) \stackrel{\nu_1^*}{\to}  {\rm Ker} \; i_0^* (A(\mathfrak g), Z)
\stackrel{\nu_2^*}{\to}  {\rm Ker} \; i_0^*(X)$, 
given by ${\rm Aut}_{(z_{\xi_1}, \ldots, z_{\xi_n})} \; \Oo_{(\epsilon_1, \ldots, \epsilon_n)}$ 
for $(\epsilon_1, \ldots, \epsilon_n)$ beeing base points for $Z/X$.  
Let $\widetilde{Y}$ be the neighborhood of $Y$ in $X$ as in Lemma \ref{t5.1}. 
Note that $CY$ is contractible. Then by Lemma \ref{l4.1} 
any twisted $A(\mathfrak g)$-bundle $\E$ 
 on $Z$ is trivial on $\widetilde{Y}_CY$. 
 Therefore 
 $p^*:  {\rm Ker}\; i_0^*(X/Y) \to {\rm Ker}\; i_0^*(Z)$   
is an isomorphism 
where $p: Z \to Z/CY =X/Y$ for the collapsing map.
The have the composition $j^*=\nu_2^*\circ p^*$ of maps. 
We then obtain the sequence
$K(A(\mathfrak g), Y)^{-1} \stackrel{\delta^{(-1)}}{\to}
 K(A(\mathfrak g), X/Y) \stackrel{j^*}{\to} {\rm Ker}\; i_0^*(X)$ given by 
${\rm Ker}\; i^*_0 ({\mathcal S}Y)^{-1} \stackrel{\delta^{(-1)}}{\to}
  {\rm Ker}\; i^*_0 (X/Y) \stackrel{j^*}{\to} {\rm Ker}\; i_0^*(X)$, 
which cohomology is 
 $H_{K_0, X/Y}={\rm Ker}\; j^*/{\rm Im}\; \delta^{(-1)} 
= {\rm Aut}_{(z_{\xi'_1}, \ldots, z_{\xi'_n})} \; \Oo_{(\epsilon'_1, \ldots, \epsilon'_n)}$. 
We apply Lemma \ref{t5.2} to approximations $(X_{C_1}Y)$ and  
$(X_{C_1} Y)_{C_2} X/(X_{C_1}Y)$ where we have denoted the cones $C_i$, $i=1$, $2$.
 Thus we obtain 
the sequence
\begin{align}\label{sequence4}
K(A(\mathfrak g), (X_{C_1}Y)_{C_2} X/ (X_{C_1} Y) ) \stackrel{\nu^*_3} {\to}  
{\rm Ker}\; i_0^* ( (X_{C_1} Y)_{C_2} X) \stackrel{\nu^*_4} {\to}  {\rm Ker}\; i_0^*( X_{C_1} Y ).  
\end{align}
One shows that this sequence is isomorphic to
 the sequence in ${\mathcal K}^{-1}_{X/Y}$ of (\ref{sequence1}).
According to the definition of $\delta$,  it is enough to will show the 
equivalence of the following two maps 
\begin{eqnarray}
\label{sequence5}
K(  A(\mathfrak g), (X_{C_1}Y)_{C_2}X/X_{C_1} Y  ) \stackrel{\nu^*_3} {\to}
 {\rm Ker}\; i_0^* ( (X_{C_1} Y)_{C_2} X ) \stackrel{\nu^*_5} {\to}
 {\rm Ker}\; i_0^*(C_2X/X) 
\\
\nonumber
= K(A(\mathfrak g), X)^{-1}  
\stackrel{i^*}{\rightarrow} K(A(\mathfrak g), Y)^{-1}={\rm Ker}\; i_0^*(C_1Y/Y).
\end{eqnarray}   
Using Lemma \ref{t5.4}, 
according to definition of $C_i$, $i=1$, $2$ we obtain the sequence 
${\rm Ker}\; i_0^*(\mathcal SY) \to 
K( A(\mathfrak g), C_1 Y/Y ) \to K(A(\mathfrak g), (C_1 Y)_{C_2} Y ) $, 
which is equivalent to 
$K(A(\mathfrak g), C_2 Y/Y)  \to {\rm Ker}\; i_0^*(\mathcal SY)$.  
Finally, we get   
\begin{eqnarray*}
&&H^0_{{\mathcal K}_{X/Y}} 
= {\rm Ker} \; \delta^{(0)}/ {\rm Im} \; \delta^{(-1)}
= H_{{\mathcal K}^0_0}
 ={\rm Aut}_{(z_1, \ldots, z_n)} \; \Oo_{X/Y}|_{\mathcal H. W}.
\end{eqnarray*}
\end{proof}

\end{document}